\documentclass[12pt]{amsart}

\usepackage{amssymb}

\newtheorem{theorem}{Theorem}[section]
\newtheorem{proposition}[theorem]{Proposition}
\newtheorem{lemma}[theorem]{Lemma}
\newtheorem{corollary}[theorem]{Corollary}
\newtheorem{definition}[theorem]{Definition}

\numberwithin{equation}{section}

\begin{document}
\baselineskip=15.5pt

\title[Parabolic $k$--ample bundles]{Parabolic $k$--ample bundles}

\author[I. Biswas]{Indranil Biswas}

\address{School of Mathematics, Tata Institute of Fundamental
Research, Homi Bhabha Road, Mumbai 400005, India}

\email{indranil@math.tifr.res.in}

\author[F. Laytimi]{Fatima Laytimi}

\address{Math\'ematiques -- B\^at. M2, Universit\'e Lille 1, F-59655
Villeneuve d'Ascq Cedex, France}

\email{Fatima.Laytimi@math.univ-lille1.fr}

\subjclass[2000]{14F05, 14F17}

\keywords{Parabolic bundle, $k$-ample bundle, projectivization,
tautological bundle}

\date{}

\begin{abstract}
We construct projectivization of a parabolic vector bundle and a
tautological line bundle over it. It is shown that a parabolic vector
bundle is ample if and only if the tautological line bundle is ample.
This allows us to generalize the notion of a $k$--ample bundle,
introduced by Sommese, to the context of parabolic bundles. A
parabolic vector bundle $E_*$ is defined
to be $k$--ample if the tautological
line bundle ${\mathcal O}_{{\mathbb P}(E_*)}(1)$ is $k$--ample.
We establish some properties of parabolic $k$--ample bundles.
\end{abstract}

\maketitle

\section{Introduction}

For an algebraic vector bundle $E$ over a complex projective
variety, let ${\mathbb P}(E)$ be the projective bundle parametrizing
hyperplanes in the fibers of $E$. The tautological quotient line
bundle on ${\mathbb P}(V)$ will be denoted by
${\mathcal O}_{{\mathbb P}(E)}(1)$. We recall that $E$ is called
ample if ${\mathcal O}_{{\mathbb P}(E)}(1)$ is ample.

Parabolic vector bundles on curves were introduced by Seshadri
in \cite{Se}. In \cite{MY}, Maruyama and Yokogawa generalized this
notion to higher dimensional varieties. In \cite{Bi2}, parabolic analog
of ample vector bundles were defined using the following standard
criterion for ample vector bundles: A vector bundle $E$ over a
complex projective variety $M$ is ample if and only if for any
vector bundle $F$ on $M$, there is an integer $n(F)$ such that
the vector bundle $S^n(E)\otimes F$
is generated by its global sections for all $n\, \geq\, n(F)$.
The Le Potier vanishing theorem for ample vector bundles and
the Hartshorne's criterion for ample vector bundles on curves were
generalized to parabolic ample bundles (see Theorem 4.4 and
Theorem 3.1 of \cite{Bi2}).

Given a parabolic vector bundle $E_*$, here we define ${\mathbb P}(E_*)$,
a parabolic analog of the projective bundle. This is done using 
the notion of ramified principal bundles (see \cite{BBN} and \cite{Bi4}
for ramified principal bundles). We also construct a tautological
line bundle ${\mathcal O}_{{\mathbb P}(E_*)}(1)$ on
${\mathbb P}(E_*)$.

In Proposition \ref{prop1} we show that
a parabolic vector bundle $E_*$ is ample if and only if the
tautological line bundle ${\mathcal O}_{{\mathbb P}(E_*)}(1)
\,\longrightarrow\, {\mathbb P}(E_*)$ is ample. Proposition
\ref{prop1} is deduced using a parabolic analog of the
above mentioned criterion for ample vector bundles in
terms of vanishing of higher cohomologies (this criterion
for parabolic vector bundles is proved in Theorem \ref{thm1}).

Sommese introduced the notion of a $k$--ample vector bundle
(see Definition 1.3 in page 232 of \cite{So}). We recall that
a line bundle $L$ over a complex projective
variety $M$ is called $k$--\textit{ample} if there is a positive
integer $t$ such that $L^{\otimes t}$ is base point free, and
all the fibers of the natural morphism $M\,\longrightarrow\,
\vert L^{\otimes t}\vert \,=\, {\mathbb P}
(H^0(M, L^{\otimes t}))$ are of dimension at most $k$.
A vector bundle $E$ is called $k$--\textit{ample} if the
tautological line bundle ${\mathcal O}_{{\mathbb P}(E)}(1)\,
\longrightarrow\,{\mathbb P}(E)$ is $k$--ample.

The constructions of the projectivization of a parabolic
bundle and the tautological
line bundle allow us to generalize the notion of $k$--ample
bundles to the context of parabolic bundles. More precisely, a
parabolic vector bundle $E_*$ is called $k$--\textit{ample} if
the tautological line bundle ${\mathcal O}_{{\mathbb P}(E_*)}(1)$
over ${\mathbb P}(E_*)$
is $k$--ample. We prove some properties of parabolic
$k$--ample bundles.

\section{Parabolic ample bundles}

Let $X$ be an irreducible smooth projective variety defined over
$\mathbb C$. Let $D\, \subset\, X$ be a simple normal crossing divisor;
this means that $D$ is effective and reduced, each irreducible component
of $D$ is smooth, and the components intersect transversally.
Let
\begin{equation}\label{decomp.-D}
D\, =\, \sum_{i=1}^\ell D_i
\end{equation}
be the decomposition of $D$ into irreducible components.
The above condition that the irreducible components of $D$
intersect transversally means that if
\begin{equation}\label{pt.-intsc.}
x\, \in\, D_{i_1}\cap D_{i_2}\cap \cdots \cap D_{i_k}\,
\subset\, D
\end{equation}
is a point where $k$ distinct components of $D$ meet, and $f_{i_j}$,
$j\,\in\, [1\, ,k]$, is the local equation of the divisor
$D_{i_j}$ around $x$, then $\{\text{d}f_{i_j}(x)\}$ is a linearly
independent subset of the holomorphic cotangent space $T^*_xX$ of
$X$ at $x$. This implies that for any choice of $k$ integers
\begin{equation}\label{k-int.}
1\, \leq \, i_1 \, <\, i_2\, <\, \cdots\, < \, i_k \, \leq\,\ell\, ,
\end{equation}
each connected component of
$D_{i_1}\bigcap D_{i_2}\bigcap \cdots \bigcap D_{i_k}$ is a smooth
subvariety of $X$.

Let $E_0$ be an algebraic vector bundle over $X$. For
each $i\,\in\, [1\, , \ell]$, let
\begin{equation}\label{divisor-filt.}
E_0\vert_{D_i}\, =\, F^i_1 \,\supsetneq\, F^i_2 \,\supsetneq\,
F^i_3 \,\supsetneq\, \cdots \,\supsetneq\, F^i_{m_i} \,\supsetneq\,
F^i_{m_i+1}\, =\, 0
\end{equation}
be a filtration by subbundles of the restriction of $E_0$ to $D_i$.
In other words, each $F^i_j$ is a subbundle of $E_0\vert_{D_i}$
with $\text{rank}(F^i_j) \, >\, \text{rank}(F^i_{j+1})$
for all $j\in [1\, , m_i]$.

A \textit{quasiparabolic} structure on $E_0$ over $D$ is a filtration
as above of each $E_0\vert_{D_i}$ satisfying the following extra
condition: Take any $k\,\in\, [1\, ,\ell]$,
and take integers $\{i_j\}_{j=1}^k$ as in \eqref{k-int.}.
If we fix some $F^{i_j}_{n_j}$, $n_j\, \in\,
[1\, , m_{i_j}]$, then over each connected component $S$
of $D_{i_1}\bigcap D_{i_2}\bigcap \cdots \bigcap D_{i_k}$,
the intersection
$$
\bigcap_{j=1}^k F^{i_j}_{n_j}\, \subset\,
E_0\vert_{D_{i_1}\cap \cdots \cap D_{i_k}}
$$
gives a subbundle of the restriction of $E_0$ to $S$.
It should be clarified that the rank of this
subbundle may depend on the choice of the connected component
$S\, \subset\,
D_{i_1}\bigcap \cdots \bigcap D_{i_k}$.

For a quasiparabolic structure as above, \textit{parabolic 
weights} are a collection of rational numbers
\begin{equation}\label{li}
0\, \leq\, \lambda^i_1\, < \, \lambda^i_2\, < \,
\lambda^i_3 \, < \,\cdots \, < \, \lambda^i_{m_i}
\, <\, 1\, ,
\end{equation}
where $i\,\in\, [1\, ,\ell]$. The parabolic weight $\lambda^i_j$
corresponds to the subbundle $F^i_j$ in \eqref{divisor-filt.}. A 
\textit{parabolic structure} on $E_0$ is a quasiparabolic
structure with parabolic weights. A vector bundle
over $X$ equipped
with a parabolic structure on it is also called a
\textit{parabolic vector bundle}. (See \cite{MY}, \cite{Se}.)

For notational convenience, a parabolic vector bundle defined
as above will be denoted by $E_*$, while $E_0$ will be referred to
as the underlying vector bundle for $E_*$. The divisor $D$ is called
the \textit{parabolic divisor} for $E_*$. We fix $D$ once and for
all, so the parabolic divisor of all parabolic vector bundles on $X$ 
considered here will be $D$.

The definitions of direct sum, tensor product and dual of vector
bundles extend naturally to parabolic vector bundles. Similarly,
symmetric and exterior powers of parabolic vector bundles are also
constructed (see \cite{MY}, \cite{Bi2}, \cite{Yo}).

We will now recall from \cite{Bi2} the definition of a parabolic
ample bundle.

A parabolic vector bundle $E_*$ is called \textit{ample} if
for every vector bundle $V\,\longrightarrow\, X$, there is
an integer $n(V)$ such that for all $n\,\geq\, n(V)$, the
vector bundle $S^n(E_*)_0\otimes V$ is generated by it global
sections, where $S^n(E_*)_0$ is the vector bundle underlying
the parabolic symmetric product $S^n(E_*)$. (See Definition 2.3
in page 514 of \cite{Bi2}.)

In Definition 2.3 of \cite{Bi2}, we take $V$ to be a coherent
sheaf. Since any coherent sheaf is a quotient of a vector bundle,
and the quotient of a vector bundle generated by global sections
is also generated by global sections, it is enough to check the
above criterion for locally free sheaves.

\begin{theorem}\label{thm1}
A parabolic vector bundle $E_*$ is ample if and only if for
every parabolic vector bundle $F_*$, there is
an integer $n(F_*)$ such that for all $n\,\geq\, n(F_*)$,
$$
H^i(X,\, (S^n(E_*)\otimes F_*)_0) \,=\, 0
$$
for all $i\,\geq \, 1$, where $(S^n(E_*)\otimes F_*)_0$
is the vector bundle underlying the parabolic tensor
product $S^n(E_*)\otimes F_*$.
\end{theorem}

\begin{proof}
First assume that $E_*$ is ample. Take any parabolic
vector bundle $F_*$ on $X$.

In \cite{Bi1}, a bijective correspondence between parabolic
bundles with parabolic weights of the form $a/a_0$, where
$a_0$ is a fixed integer, and orbifold vector bundles is 
established. According
to it, there is a finite (ramified) algebraic
Galois covering
\begin{equation}\label{e1}
\gamma\, :\, Y\,\longrightarrow\, X\, ,
\end{equation}
where $Y$ is an irreducible smooth projective variety
with the following property: there are $\text{Gal}
(\gamma)$--linearized vector bundles $E'$ and
$F'$ over $Y$ that correspond to the parabolic vector
bundles $E_*$ and $F_*$ respectively.

For notational convenience, the Galois group
$\text{Gal}(\gamma)$ will be denoted by $\Gamma$.

Since $E_*$ is parabolic ample, the vector bundle $E'$ is ample
(see Lemma 4.6 in page 522 of \cite{Bi2}). As $E'$ is ample,
there is an integer $m_0$ such that for all $n\,\geq \, m_0$,
\begin{equation}\label{a1}
H^i(Y, \, S^n(E')\otimes F')\,=\, 0
\end{equation}
for all $i\,\geq \, 1$ (see Proposition (3.3) in page
70 of \cite{Ha}).
The map $\gamma$ being finite,
for any vector bundle $W$ on $Y$, we have a canonical
identification $H^j(Y,\, W)\,
= \, H^j(X,\, \gamma_*W)$. Therefore, from \eqref{a1},
\begin{equation}\label{e2}
H^i(X, \, \gamma_*(S^n(E')\otimes F'))\,=\, 0
\end{equation}
for all $i\,\geq \, 1$ and all $n\,\geq \, m_0$.

The above mentioned correspondence between parabolic
vector bundles and $\Gamma$--linearized vector bundles
takes the tensor product of any two $\Gamma$--linearized vector
bundles to the parabolic tensor product of the corresponding
parabolic vector bundles. Therefore, the
parabolic vector bundle over $X$ corresponding to
the $\Gamma$--linearized vector bundle $S^n(E')\otimes F'$
is the parabolic tensor product
$S^n(E_*)\otimes F_*$. Hence
$$
(\gamma_*(S^n(E')\otimes F'))^\Gamma \,=\, 
(S^n(E_*)\otimes F_*)_0
$$
(see (2.9) in page 310 of \cite{Bi1}), where
\begin{equation}\label{e3}
(\gamma_*(S^n(E')\otimes F'))^\Gamma
\, \subset\, \gamma_*(S^n(E')\otimes F')
\end{equation}
is the sheaf of $\Gamma$--invariants. We note that the
subsheaf in \eqref{e3} is a direct summand (see
the proof of Proposition 4.10 in page 525 of \cite{Bi2}).
Therefore, from \eqref{e2} it follows that
\begin{equation}\label{f1}
H^i(X, \, (S^n(E_*)\otimes F_*)_0)\,=\, 0
\end{equation}
for all $n\,\geq \, m_0$ and all $i\,\geq \, 1$.

To prove the converse, assume that
given any parabolic vector bundle $F_*$, there is
an integer $n(F_*)$ such that for all $n\,\geq\, n(F_*)$,
\begin{equation}\label{e5}
H^i(X,\, (S^n(E_*)\otimes F_*)_0) \,=\, 0
\end{equation}
for all $i\,\geq \, 1$. We will prove that $E_*$ is
parabolic ample.

Take any coherent sheaf $W$ on $X$. We will show that
\begin{equation}\label{e4}
H^i(X,\, S^n(E_*)_0\otimes W)\,=\, 0
\end{equation}
for all $i\,\geq\, 1$ and all $n$ sufficiently large, where
$S^n(E_*)_0$ is the vector bundle underlying the parabolic
symmetric product $S^n(E_*)$.

The coherent sheaf $W$ is a quotient of some locally free coherent
sheaf. Let $V$ be a locally free sheaf on $X$, and let
$$
\rho\, :\, V\,\longrightarrow\, W
$$
be a surjective homomorphism. The kernel of $\rho$ will
be denoted by $\mathcal K$. So we have a short exact sequence
\begin{equation}\label{e6}
0\,\longrightarrow\, S^n(E_*)_0\otimes {\mathcal K} \,
\longrightarrow\,S^n(E_*)_0\otimes V \,\stackrel{\text{Id}
\otimes\rho}{\longrightarrow}\,(S^n(E_*)_0\otimes W
\,\longrightarrow\, 0
\end{equation}
(since $(S^n(E_*)_0$ is locally free, tensoring with it preserves
exactness of a sequence).
Equip $V$ with the trivial parabolic structure (so it
has no nonzero parabolic weight). Consequently, $(S^n(E_*)\otimes V)_0
\,=\, S^n(E_*)_0\otimes V$. Invoking the given condition
in \eqref{e5}, from the long exact sequence of cohomologies
associated to the short exact sequence in
\eqref{e6} we conclude that given any $i\,\geq\, 1$,
$$
H^i(X,\, S^n(E_*)_0\otimes W)\,=\, 0
$$
for all $n$ sufficiently large if $H^{i+1}(X,\, S^m(E_*)_0\otimes
{\mathcal K})\,=\, 0$ for all $m$ sufficiently large. Using this
inductively, and noting that all cohomologies of degree larger
than $\dim X$ vanish, we conclude that \eqref{e4} holds.

For any point $x\,\in\, X$, let ${\mathcal I}_x\,\subset\, 
{\mathcal O}_X$ be the ideal sheaf. For any vector bundle
$F$ on $X$, consider the short exact sequence of coherent sheaves
$$
0\,\longrightarrow\, S^n(E_*)_0\otimes F\otimes{\mathcal I}_x \,
\longrightarrow\, S^n(E_*)_0\otimes F \,\longrightarrow
\,((S^n(E_*)_0\otimes F)_x \,\longrightarrow\, 0\, .
$$
Using the above observation that
$H^1(X,\, S^n(E_*)_0\otimes F\otimes{\mathcal I}_x)\,=\, 0$
for all $n$ sufficiently large, from the long exact sequence of
cohomologies associated to this short exact sequence we conclude
that for all $n$ sufficiently large,
the fiber $((S^n(E_*)_0\otimes F)_x$ is generated by the
global sections of $S^n(E_*)_0\otimes F$. Now
using the Noetherian property of $X$ it follows that
for all $n$ sufficiently large, the vector
bundle $S^n(E_*)_0\otimes F$ is generated by its global sections
(see Proposition 2.1 in page 65 of \cite{Ha}).
Therefore, the parabolic vector bundle $E_*$ is ample.
\end{proof}

\section{Ramified bundles and ampleness}

\subsection{Ramified bundles and parabolic bundles}
The complement of $D$ in $X$ will be denoted by $X-D$;
we will avoid the more standard notation $X\setminus D$
in order to avoid any confusion with left quotient which
we will frequently need.

Let
$$
\varphi\, :\, E_{\text{GL}(r, {\mathbb C})}\, \longrightarrow\, X
$$
be a ramified principal $\text{GL}(r, {\mathbb C})$--bundle
with ramification over $D$
(see \cite{BBN}, \cite{Bi3}, \cite{Bi4} for the definition). We
briefly recall its defining properties.
The total space $E_{\text{GL}(r, {\mathbb C})}$ is a
smooth complex variety equipped with
an algebraic right action of $\text{GL}(r, {\mathbb C})$
$$
f\, :\,E_{\text{GL}(r, {\mathbb C})}\times \text{GL}(r,
{\mathbb C})\, \longrightarrow\, E_{\text{GL}(r, {\mathbb C})}\, ,
$$
and the following conditions hold:
\begin{enumerate}
\item{} $\varphi\circ f \, =\, \varphi\circ p_1$, where $p_1$ is
the natural projection of $E_{\text{GL}(r, {\mathbb C})}\times
\text{GL}(r, {\mathbb C})$ to $E_{\text{GL}(r, {\mathbb C})}$,

\item{} for each point $x\, \in\, X$, the action of
$\text{GL}(r, {\mathbb C})$ on the
reduced fiber $\varphi^{-1}(x)_{\text{red}}$ is transitive,

\item{} the restriction of $\varphi$ to $\varphi^{-1}(X
- D)$ makes $\varphi^{-1}(X- D)$ a principal
$\text{GL}(r, {\mathbb C})$--bundle over $X- D$,

\item{} for each irreducible component $D_i\, \subset\, D$,
the reduced inverse image $\varphi^{-1}(D_i)_{\text{red}}$ is a
smooth divisor and
$$
\widehat{D}\, :=\, \sum_{i=1}^\ell \varphi^{-1}(D_i)_{\text{red}}
$$
is a normal crossing divisor on $E_{\text{GL}(r, {\mathbb C})}$, and

\item{} for any point $x$ of $D$, and any point
$z\, \in\, \varphi^{-1}(x)$, the isotropy
group
\begin{equation}\label{e8}
G_z\, \subset\,\text{GL}(r, {\mathbb C}) \, ,
\end{equation}
for the action of $\text{GL}(r, {\mathbb C})$ on
$E_{\text{GL}(r, {\mathbb C})}$, is a finite group, and if
$x$ is a smooth point of $D$, then the natural action of 
on the quotient line $T_zE_{\text{GL}(r,
{\mathbb C})}/T_z\varphi^{-1}(D)_{\text{red}}$ is faithful.
\end{enumerate}

We will note some some properties of the isotropy subgroup $G_z$ in 
condition (5) of the above list. Let
$$
D^{\rm sm}\, \subset\, D
$$
be the smooth locus of the divisor. Take any $x\, \in\,D^{\rm sm}$,
and take any $z\, \in\, \varphi^{-1}(x)$.
The group of linear automorphisms of the complex line
$T_zE_{\text{GL}(r,
{\mathbb C})}/T_z\varphi^{-1}(D)_{\text{red}}$ in condition (5)
is ${\mathbb C}^*$, and any finite
subgroup of ${\mathbb C}^*$ is a cyclic group. Therefore, we
conclude that the isotropy subgroup $G_z$ is a
finite cyclic group if $x\, \in\,D^{\rm sm}$ (because $G_z$ acts 
faithfully on
$T_zE_{\text{GL}(r, {\mathbb C})}/T_z\varphi^{-1}(D)_{\text{red}}$).
Take any $z'\,\in\, E_{\text{GL}(r, {\mathbb C})}$ such that
$\varphi(z')\, =\, \varphi(z)$. There is an element $g\,\in\,
\text{GL}(r, {\mathbb C})$ such that $f(z\, ,g)\,=\, z'$. Therefore,
$G_z$ is isomorphic to $G_{z'}$. As $z$ moves over a fixed
connected component of $\varphi^{-1}(D^{\rm sm})$, the order
of the finite cyclic group $G_z$ remains unchanged. Fix an
integer $m\, \geq\, 2$. Let
$$
S(m) \, \subset\, \text{GL}(r, {\mathbb C})
$$
be the subset of all elements of order exactly $m$. This
$S(m)$ is a disjoint union of finitely many closed orbits
for the the adjoint action of $\text{GL}(r, {\mathbb C})$
on itself. Therefore, as $z$ moves over a fixed
connected component of $\varphi^{-1}(D^{\rm sm})$, the conjugacy
class of the subgroup $G_z\, \subset\, \text{GL}(r, {\mathbb C})$
remains unchanged.

There is a natural bijective correspondence between the ramified
principal $\text{GL}(r, {\mathbb C})$--bundles with ramification
over $D$ and the parabolic vector bundles
of rank $r$ with $D$ as the parabolic divisor (see
\cite{BBN}, \cite{Bi3}). This correspondence is recalled below.

Take the standard $\text{GL}(r, {\mathbb C})$--module ${\mathbb C}^r$;
for notational convenience, we will denote this $\text{GL}(r, 
{\mathbb C})$--module by $V$. Recall that
$$
E^0_{\text{GL}(r, {\mathbb C})}\, :=\, \varphi^{-1}(X- D)
\,\longrightarrow\,X- D
$$
is a usual principal $\text{GL}(r, {\mathbb C})$--bundle. Let
$$
E^0_V \,=\,E^0_{\text{GL}(r, {\mathbb C})}(V) \,:=\,
E^0_{\text{GL}(r, {\mathbb C})}\times^{\text{GL}(r, {\mathbb C})} V
\,\longrightarrow\, X- D
$$
be the associated vector bundle. This vector bundle has a natural
extension to $X$ as a parabolic vector bundle (its construction
is similar to the construction of a parabolic vector bundle from
an orbifold vector bundle; see \cite{Bi1}). Therefore,
starting from a ramified principal $\text{GL}(r,
{\mathbb C})$--bundle $\varphi\, :\, E_{\text{GL}(r, {\mathbb 
C})}\,\longrightarrow\, X$ we get a
parabolic vector bundles over $X$ with parabolic structure over $D$.

We will give an alternative description of the correspondence.

Let $E_*$ be a parabolic vector bundle of rank $r$ over $X$.
Let $E'\,\longrightarrow\, Y$ be the corresponding
$\text{Gal}(\gamma)$--linearized vector bundle, where $\gamma$ is the
Galois covering in \eqref{e1}. As before, $\text{Gal}(\gamma)$
will be denoted by $\Gamma$. Let $$E'_{\text{GL}(r, {\mathbb C})}
\,\longrightarrow\, Y$$
be the principal $\text{GL}(r, {\mathbb C})$--bundle defined by $E'$.
So the fiber of $E'_{\text{GL}(r, {\mathbb C})}$ over any point
$y\,\in\, Y$ is the space of all
linear isomorphisms from ${\mathbb C}^r$ to the fiber
$E'_y$. The $\Gamma$--linearization of $E'$ produces a left action
of the Galois group $\Gamma$ on the total space of $E'_{\text{GL}
(r, {\mathbb C})}$. The actions of $\text{GL}(r, {\mathbb C})$ and 
$\Gamma$ on $E'_{\text{GL}(r, {\mathbb C})}$ commute, so the
quotient $\Gamma\backslash E'_{\text{GL}(r, {\mathbb C})}$
for the action of $\Gamma$ on $E'_{\text{GL}(r, {\mathbb C})}$ has an
action of $\text{GL}(r, {\mathbb C})$. The projection of
$E'_{\text{GL}(r, {\mathbb C})}$ to $Y$ produces a projection of
$\Gamma\backslash E'_{\text{GL}(r, {\mathbb C})}$ to
$\Gamma\backslash Y\, =\, X$. This resulting morphism
$$
\Gamma\backslash E'_{\text{GL}(r, {\mathbb C})}
\,\longrightarrow\, X
$$
defines a ramified principal $\text{GL}(r, {\mathbb C})$--bundle.

Conversely, if $F_{\text{GL}(r, {\mathbb C})}\, \longrightarrow\,
X$ is a ramified principal $\text{GL}(r, {\mathbb C})$--bundle,
then there is a finite (ramified) Galois covering
$$
\gamma\, :\, Y\, \longrightarrow\, X
$$
such that the normalization $\widetilde{F_{\text{GL}(r, {\mathbb C})}
\times_X Y}$ of the fiber product
$F_{\text{GL}(r, {\mathbb C})}\times_X Y$ is smooth. The projection 
$\widetilde{F_{\text{GL}(r, {\mathbb C})}\times_X Y}
\,\longrightarrow\, Y$ is a principal $\text{GL}(r,
{\mathbb C})$--bundle equipped with an action of the Galois
group $\Gamma\,:=\,
\text{Gal}(\gamma)$. Let $F_V\,:=\, \widetilde{F_{\text{GL}(r, 
{\mathbb C})}\times_X Y}(V)$ be the vector bundle over
$Y$ associated to the principal $\text{GL}(r,
{\mathbb C})$--bundle $\widetilde{F_{\text{GL}(r, {\mathbb C})}
\times_X Y}$ for the standard $\text{GL}(r, {\mathbb C})$--module
$V \,:=\, {\mathbb C}^r$. The action of $\Gamma$
on $Y$ induces an action of $\Gamma$ on $\widetilde{F_{\text{GL}(r,
{\mathbb C})}\times_X Y}$; this action of $\Gamma$ on
$\widetilde{F_{\text{GL}(r, {\mathbb C})}\times_X Y}$ 
commutes with the action of
${\text{GL}(r, {\mathbb C})}$ on $\widetilde{F_{\text{GL}(r, 
{\mathbb C})}\times_X Y}$. Hence the action of $\Gamma$ on
$\widetilde{F_{\text{GL}(r, {\mathbb C})}\times_X Y}$ induces 
an action of $\Gamma$ on the above defined associated bundle
$F_V$ making $F_V$ a $\Gamma$--linearized vector
bundle. Let $E_*$ be the parabolic vector bundle of rank $r$
over $X$ associated
to this $\Gamma$--linearized vector bundle $F_V$.

The above construction of a parabolic vector bundle of rank
$r$ from a ramified principal ${\text{GL}(r, {\mathbb C})}
$--bundle is the inverse of the earlier construction
of a ramified principal ${\text{GL}(r, {\mathbb C})}
$--bundle from a parabolic vector bundle.

\subsection{Projectivization and the tautological line bundle}

Let $E_*$ be a parabolic vector bundle over $X$ of rank $r$. Let
\begin{equation}\label{ee}
\varphi\, :\, E_{\text{GL}(r, {\mathbb C})}\,\longrightarrow\, X
\end{equation}
be the corresponding ramified principal
$\text{GL}(r, {\mathbb C})$--bundle. Let ${\mathbb P}^{r-1}$ be
the projective space parametrizing the
hyperplanes in ${\mathbb C}^r$. The standard
action of $\text{GL}(r, {\mathbb C})$ on ${\mathbb C}^r$ produces
an action of $\text{GL}(r, {\mathbb C})$ on ${\mathbb P}^{r-1}$. Let 
\begin{equation}\label{e7}
{\mathbb P}(E_*) \, = \, E_{\text{GL}(r, {\mathbb C})}({\mathbb P}^{r-1})
\, :=\, E_{\text{GL}(r, {\mathbb C})}\times^{\text{GL}(r,
{\mathbb C})} {\mathbb P}^{r-1}\,\longrightarrow\, X
\end{equation}
be the associated (ramified) fiber bundle.

\begin{definition}\label{def1}
{\rm We will call ${\mathbb P}(E_*)$ the} projective bundle {\rm
associated to the parabolic vector bundle $E_*$.}
\end{definition}

Take a point $x\,\in\,D$; it should be
clarified that $x$ need not be a smooth point of $D$.
Take any $z\, in\, \varphi^{-1}(x)$, where $\varphi$
is the morphism in \eqref{ee}. As in \eqref{e8}, let
$G_z\,\subset\, \text{GL}(r, {\mathbb C})$
be the isotropy subgroup for $z$
for the action of $\text{GL}(r, {\mathbb C})$ on $E_{\text{GL}(r, 
{\mathbb C})}$. We recall that $G_z$ is a finite group. Let $n_x$
be the order of $G_z$ (it is independent of the choice of $z$ because
$\text{GL}(r, {\mathbb C})$ acts transitively on
the fiber $\varphi^{-1}(x)$). The number of
distinct integers $n_x$, $x\,\in\,D$, is finite.
Let
\begin{equation}\label{e9}
N(E_*)\, :=\, \text{l.c.m.}\{n_x\}_{x\in D}
\end{equation}
be the least common multiple of all these integers $n_x$.

As before, ${\mathbb P}^{r-1}$ is the projective
space parametrizing the hyperplanes in ${\mathbb C}^r$.
For any point $y\,\in\, {\mathbb P}^{r-1}$, let
\begin{equation}\label{e10}
H_y\,\subset\, \text{GL}(r, {\mathbb C})
\end{equation}
be the isotropy subgroup for the action of $\text{GL}(r, {\mathbb C})$
on ${\mathbb P}^{r-1}$ constructed using the standard action of
$\text{GL}(r, {\mathbb C})$ on ${\mathbb C}^r$. So $H_y$ is
a maximal parabolic subgroup of $\text{GL}(r, {\mathbb C})$.
Let ${\mathcal O}_{{\mathbb P}^{r-1}}(1)\,\longrightarrow \,
{\mathbb P}^{r-1}$ be the tautological quotient line bundle. The group
$H_y$ in \eqref{e10} acts on the fiber of
${\mathcal O}_{{\mathbb P}^{r-1}}(1)$ over the point $y$.

{}From the definition of $N(E_*)$ in \eqref{e9} it follows for any
$z\,\in\, \varphi^{-1}(D)$ and any $y\,\in\, {\mathbb P}^{r-1}$,
the group $G_z\bigcap H_y\,\subset\, \text{GL}(r, {\mathbb C})$
acts trivially on the fiber of the line bundle
$${\mathcal O}_{{\mathbb P}^{r-1}}(N(E_*))\, :=\,
{\mathcal O}_{{\mathbb P}^{r-1}}(1)^{\otimes N(E_*)}
$$
over the point $y$; the group $G_z$ is defined in \eqref{e8}.

Consider the action of $\text{GL}(r, {\mathbb C})$ on the total space
of the line bundle ${\mathcal O}_{{\mathbb P}^{r-1}}
(N(E_*))$ constructed using the standard action of
$\text{GL}(r, {\mathbb C})$ on ${\mathbb C}^r$. Let
$$
E_{\text{GL}(r, {\mathbb C})}({\mathcal O}_{{\mathbb P}^{r-1}}
(N(E_*)))
\, :=\, E_{\text{GL}(r, {\mathbb C})}\times^{\text{GL}(r, {\mathbb C})}
{\mathcal O}_{{\mathbb P}^{r-1}} 
(N(E_*)) \,\longrightarrow\, X
$$
be the associated fiber bundle. Since the natural projection
$$
{\mathcal O}_{{\mathbb P}^{r-1}} (N(E_*))\,\longrightarrow\,
{\mathbb P}^{r-1}
$$
intertwines the actions of $\text{GL}(r, {\mathbb C})$ on
${\mathcal O}_{{\mathbb P}^{r-1}} (N(E_*))$ and ${\mathbb P}^{r-1}$,
it produces a projection between the associated bundles
\begin{equation}\label{e11}
E_{\text{GL}(r, {\mathbb C})}({\mathcal O}_{{\mathbb P}^{r-1}}
(N(E_*))) \,\longrightarrow\, {\mathbb P}(E_*)\, ,
\end{equation}
where ${\mathbb P}(E_*)$ is the associated bundle constructed
in \eqref{e7}.

Using the above observation that $G_z\bigcap H_y$ acts trivially on
the fiber of ${\mathcal O}_{{\mathbb P}^{r-1}} (N(E_*))$ over $y$ it
follows immediately that $E_{\text{GL}(r, {\mathbb C})}
({\mathcal O}_{{\mathbb P}^{r-1}} (N(E_*)))$ in \eqref{e11} is
an algebraic line bundle over the variety ${\mathbb P}(E_*)$.

\begin{definition}\label{def2}
{\rm The line bundle $E_{\text{GL}(r, {\mathbb C})}({\mathcal O}_{
{\mathbb P}^{r-1}}(N(E_*)))\,\longrightarrow\, {\mathbb P}(E_*)$ will be
called the} tautological line bundle; {\rm this tautological line bundle
will be denoted by ${\mathcal O}_{{\mathbb P}(E_*)}(1)$.}
\end{definition}

For any positive integer $m$, the line bundle ${\mathcal
O}_{{\mathbb P}(E_*)}(1)^{\otimes m}$ (respectively,
$({\mathcal O}_{{\mathbb P}(E_*)}(1)^{\otimes m})^*$)
will be denoted by ${\mathcal O}_{{\mathbb P}(E_*)}(m)$
(respectively, ${\mathcal O}_{{\mathbb P}(E_*)}(-m)$).
Also, ${\mathcal O}_{{\mathbb P}(E_*)}(0)$ will denote
the trivial line bundle.

\begin{proposition}\label{prop1}
A parabolic vector bundle $E_*$ is ample if and only if the
tautological line bundle ${\mathcal O}_{{\mathbb P}(E_*)}(1)
\,\longrightarrow\, {\mathbb P}(E_*)$ is ample.
\end{proposition}

\begin{proof}
First assume that the parabolic vector bundle $E_*$ is ample.
Let $E'\,\longrightarrow\, Y$ be the corresponding
$\text{Gal}(\gamma)$--linearized vector bundle over $Y$, where
$\gamma$ is the covering in \eqref{e1}. The vector bundle
$E'$ is ample because $E_*$ is parabolic ample (Lemma 4.6 in
page 522 of \cite{Bi2}). Hence the tautological quotient line
bundle ${\mathcal O}_{{\mathbb P}(E')}(1)\,\longrightarrow\,
{\mathbb P}(E')$ is ample, where ${\mathbb P}(E')$ is the
space of all hyperplanes in the fibers of $E'$. Since
${\mathcal O}_{{\mathbb P}(E')}(1)$ is ample, the line bundle
$$
{\mathcal O}_{{\mathbb P}(E')}(N(E_*))\,\longrightarrow\,
{\mathbb P}(E')
$$
is ample, where $N(E_*)$ is defined in \eqref{e9}.

The action of $\text{Gal}(\gamma)$ on $E'$ produces
a left action of $\text{Gal}(\gamma)$ on
${\mathbb P}(E')$.
Evidently, the variety ${\mathbb P}(E_*)$ in \eqref{e7} is the
quotient
\begin{equation}\label{g1}
\text{Gal}(\gamma)\backslash {\mathbb P}(E')\,=\,
{\mathbb P}(E_*)\, .
\end{equation}
We note that the quotient $\text{Gal}(\gamma)\backslash
{\mathcal O}_{{\mathbb P}(E')}(N(E_*))$ is a line bundle
over $\text{Gal}(\gamma)\backslash {\mathbb P}(E')$ because
the isotropy
subgroups, for the action of $\text{Gal}(\gamma)$ on
${\mathbb P}(E')$, act trivially on the corresponding fibers
of ${\mathcal O}_{{\mathbb P}(E')}(N(E_*))$. We have a
natural isomorphism of line bundles
\begin{equation}\label{g2}
\text{Gal}(\gamma)\backslash
{\mathcal O}_{{\mathbb P}(E')}(N(E_*))\,=\,
{\mathcal O}_{{\mathbb P}(E_*)}(1)\, .
\end{equation}
Since the line bundle
${\mathcal O}_{{\mathbb P}(E')}(N(E_*))$ is ample, from \eqref{g2}
it follows that the line bundle
${\mathcal O}_{{\mathbb P}(E_*)}(1)$ is ample.

To prove the converse, assume that the
tautological line bundle ${\mathcal O}_{{\mathbb P}(E_*)}(1)$
is ample. Take any parabolic vector bundle $F_*$ on $X$. As in the
first part of the proof of Theorem \ref{thm1}, take the covering
$\gamma$ in \eqref{e1} such that there is a $\text{Gal}
(\gamma)$--linearized vector bundle $E'$ (respectively,
$F'$) over $Y$ that corresponds to the parabolic vector
bundle $E_*$ (respectively, $F_*$). We will use the
criterion in Theorem \ref{thm1} to show that the
parabolic vector bundle $E_*$ is ample.

As the quotient map ${\mathbb P}(E') \,\longrightarrow\,
\text{Gal}(\gamma)\backslash {\mathbb P}(E')$ is a finite
morphism, the pullback of ${\mathcal O}_{{\mathbb P}(E_*)}(1)$
to ${\mathbb P}(E')$ is ample. Since this pullback is
${\mathcal O}_{{\mathbb P}(E')}(N(E_*))$ (see \eqref{g2}), we
conclude that the vector bundle $E'$ is ample.

Take any $i\,\geq\,1$. We have
$$
H^i(Y, \, S^n(E')\otimes F')\,=\, 0
$$
for all $n$ sufficiently large, because $E'$ is ample.
Since $\gamma$ is a finite map, this implies that
$$
H^i(X, \, \gamma_*(S^n(E')\otimes F'))\,=\,
H^i(Y, \, S^n(E')\otimes F')\,=\, 0
$$
for all $n$ sufficiently large. Now
$$
H^i(X, \, (S^n(E_*)\otimes F_*)_0)\,=\, 0
$$
for all $n$ sufficiently large because
$(S^n(E_*)\otimes F_*)_0$ is a direct summand of
$\gamma_*(S^n(E')\otimes F')$ (see \eqref{f1}).
Hence $E_*$ is ample by Theorem \ref{thm1}.
\end{proof}

\section{Parabolic $k$--ample bundles}

We recall from \cite{So} the definition of a $k$--ample vector
bundle. A line bundle $L$ over a complex projective
variety $M$ is called $k$--\textit{ample} if there is a positive
integer $t$ such that
\begin{itemize}
\item{} $L^{\otimes t}$ is base point free, and

\item all the fibers of the natural morphism $M\,\longrightarrow\,
{\mathbb P}(H^0(M, L^{\otimes t}))$ are of dimension at most $k$.
\end{itemize}
A vector bundle $E$ is called $k$--\textit{ample} if the
tautological line bundle ${\mathcal O}_{{\mathbb P}(E)}(1)\,
\longrightarrow\,{\mathbb P}(E)$ is $k$--ample.
(See Definition 1.3 in page 232 of \cite{So}.)

\begin{definition}\label{def3} 
{\rm A parabolic vector bundle $E_*$ over $X$ will be called}
$k$--ample {\rm if the tautological line bundle
${\mathcal O}_{{\mathbb P}(E_*)}(1)\,\longrightarrow\,
{\mathbb P}(E_*)$ (see Definition \ref{def2}) is $k$--ample.}
\end{definition}

We will express the cohomology of
a ${\rm Gal}(\gamma)$--linearized vector bundle in terms
of cohomology of some parabolic vector bundle.

Let $E_*$ be a parabolic vector bundle on $X$, and let
$E' \,\longrightarrow \, Y$ be the corresponding
${\rm Gal}(\gamma)$--linearized vector bundle, where
$\gamma\, :\, Y \,\longrightarrow \, X$ is a covering
as in \eqref{e1}. Let ${\mathbb C}[\Gamma]$ be the group
algebra of the finite group $\Gamma \, :=\, {\rm Gal}(\gamma)$.
The diagonal action of $\Gamma$ on $Y\times {\mathbb C}[\Gamma]$
constructed using the left action of $\Gamma$ on ${\mathbb C}[\Gamma]$
makes the trivial vector bundle
$$
{\mathcal V}' \, := \, Y\times {\mathbb C}[\Gamma]\,
\longrightarrow \, Y
$$
a $\Gamma$--linearized vector bundle. Let
$$
{\mathcal V}_*\,\longrightarrow \, X
$$

be the parabolic vector bundle $X$ corresponding to
this $\Gamma$--linearized vector bundle ${\mathcal V}'$.

\begin{lemma}\label{lem1}
There is a natural isomorphism
$$
\gamma_* E' \, \stackrel{\sim}{\longrightarrow}\,
(E_*\otimes {\mathcal V}_*)_0\, ,
$$
where $(E_*\otimes {\mathcal V}_*)_0$ is the vector
bundle underlying the parabolic tensor product
$E_*\otimes {\mathcal V}_*$.
\end{lemma}

\begin{proof}
The group algebra ${\mathbb C}[\Gamma]$ is a $\Gamma$--bimodule.
For any $\Gamma$--module $M_0$, consider the homomorphism
of $\Gamma$--modules
$$
{\mathbb C}[\Gamma]\otimes_{\mathbb C} M_0\,\longrightarrow\,
M_0
$$
defined by the action of $\Gamma$ on $M_0$. The composition
$$
({\mathbb C}[\Gamma]\otimes_{\mathbb C} M_0)^\Gamma
\,\hookrightarrow\,{\mathbb C}[\Gamma]\otimes_{\mathbb C} M_0
\,\longrightarrow\, M_0 
$$
is an isomorphism, where $({\mathbb C}[\Gamma]\otimes_{\mathbb C}
M_0)^\Gamma\, \subset\, {\mathbb C}[\Gamma]\otimes_{\mathbb C}M_0$
is the space of all $\Gamma$--invariants. Using 
this, the lemma follows immediately from the definition of
$(E_*\otimes {\mathcal V}_*)_0$.
\end{proof}

\begin{corollary}\label{cor1}
There is a natural isomorphism
$$
H^i(Y,\, E')\, \stackrel{\sim}{\longrightarrow}\,
H^i(X,\, (E_*\otimes {\mathcal V}_*)_0)
$$
for every $i\,\geq\, 0$.
\end{corollary}

\begin{proof}
Since $\Gamma$ is a finite map, we have
$$
H^i(Y,\, E')\,=\, H^i(X,\, \gamma_* E')\, .
$$
Now Lemma \ref{lem1} completes the proof.
\end{proof}

\begin{corollary}\label{cor2}
Take any vector bundle $W$ on $Y$. Let
$$
\widetilde{W}\, :=\, \bigoplus_{g\in \Gamma}
g^*W
$$
be the direct sum of all translates of $W$ by the
automorphisms of $Y$ lying in the Galois group. Consider the
natural $\Gamma$--linearization on $\widetilde{W}$. Let
$F_*$ be the parabolic vector bundle on $X$ corresponding to
the $\Gamma$--linearized vector bundle $\widetilde{W}$.
There is a natural isomorphism
$$
H^i(Y,\, \widetilde{W})\, \stackrel{\sim}{\longrightarrow}\,
H^i(X,\, (F_*\otimes {\mathcal V}_*)_0)
$$
for every $i\,\geq\, 0$. In particular,
$$
H^j(Y,\, W)\,=\, 0
$$
if $H^j(X,\, (F_*\otimes {\mathcal V}_*)_0)\,=\, 0$.
\end{corollary}

\begin{proof}
The first statement in the corollary is a special case of
Corollary \ref{cor1}: set the parabolic vector bundle $E_*$
in Corollary \ref{cor1} to be $F_*$. The
second part of the corollary follows from the first
part because $W$ is a direct summand of $\widetilde{W}$.
\end{proof}

The following proposition is the parabolic analog of
Proposition 1.7 in page 233 of \cite{So}.

\begin{proposition}\label{prop2}
Let $E_*$ be parabolic vector bundle on $X$ such that the line
bundle ${\mathcal O}_{{\mathbb P}(E_*)}(t)$ is base point
free for some positive integer $t$. Then $E_*$ is $k$--ample
ample if and only if for
every parabolic vector bundle $F_*$ on $X$, there is
an integer $n(F_*)$ such that for all $n\,\geq\, n(F_*)$,
$$
H^i(X,\, (S^n(E_*)\otimes F_*)_0) \,=\, 0
$$
for all $i\,\geq \, k+1$, where $(S^n(E_*)\otimes F_*)_0$
is the vector bundle underlying the parabolic tensor
product $S^n(E_*)\otimes F_*$.
\end{proposition}

\begin{proof}
First assume that $E_*$ is $k$--ample. Let
$$
\psi\, :\,
{\mathbb P}(E')\,\longrightarrow\, \text{Gal}(\gamma)\backslash
{\mathbb P}(E') \,=\, {\mathbb P}(E_*)
$$
be the quotient map; see \eqref{g1}. As $E_*$ is $k$--ample,
the line bundle ${\mathcal O}_{{\mathbb P}(E_*)}(1)$ is, by
definition, $k$--ample. Since $\psi$ is a finite
morphism, the pullback $\psi^*{\mathcal O}_{{\mathbb P}(E_*)}(1)$
is also $k$--ample. But $\psi^*{\mathcal O}_{{\mathbb P}(E_*)}(1)
\,=\,{\mathcal O}_{{\mathbb P}(E')}(N(E_*))$ (see \eqref{g2}),
so ${\mathcal O}_{{\mathbb P}(E')}(N(E_*))$ is 
$k$--ample. Hence the vector bundle $E'$ over $Y$ is $k$--ample.

Let $F_*$ be a parabolic vector bundle on $X$, and let
$F'$ be the corresponding $\Gamma$--linearized vector
bundle on $Y$. Since $E'$ is $k$--ample, there is
an integer $n(E')$ such that for all $n\,\geq\, n(F')$,
$$
H^i(Y,\, S^n(E')\otimes F') \,=\, 0
$$
for all $i\,\geq \, k+1$ (Proposition 1.7 in page 233 of \cite{So}).

We have already seen in the proof of Theorem \ref{thm1}
(also in the proof of Proposition \ref{prop1}) that
$H^i(X,\, (S^n(E_*)\otimes F_*)_0)$ is a subspace of
$H^i(Y,\, S^n(E')\otimes F')$. Therefore, we conclude that
$$
H^i(X,\, (S^n(E_*)\otimes F_*)_0) \,=\, 0
$$
for all $i\,\geq \, k+1$ and all $n\,\geq\, n(F')$.

To prove the converse, assume that
for
every parabolic vector bundle $F_*$, there is
an integer $n(F_*)$ such that for all $n\,\geq\, n(F_*)$,
$$
H^i(X,\, (S^n(E_*)\otimes F_*)_0) \,=\, 0
$$
for all $i\,\geq \, k+1$.

In view of \eqref{g1} and \eqref{g2}, to prove that
the parabolic vector bundle $E_*$ is $k$--ample, it suffices to
show that the corresponding vector bundle
$E'\,\longrightarrow\, Y$ is $k$--ample.

Take any vector bundle $V\,\longrightarrow\, Y$. Set
the vector bundle $W$ in Corollary \ref{cor2} to be
$S^n(E')\otimes V$. From the second part of
Corollary \ref{cor2}, together with the given condition
on $E_*$, we conclude that there is 
an integer $n(V)$ such that
$$
H^i(Y,\, S^n(E')\otimes V) \,=\, 0
$$
for all $i\,\geq \, k+1$ and $n\,\geq\, n(V)$. Now from
Proposition 1.7 in page 233 of
\cite{So} it follows that the vector bundle $E'$ is $k$--ample.
\end{proof}

We can translate the various vanishing theorems on
$k$--ample vector bundles to parabolic $k$--ample bundles.
More precisely, using the correspondence between parabolic
vector bundles and $\Gamma$--linearized vector bundles, any
vanishing result on $k$--ample vector bundles will give a
corresponding vanishing result on parabolic $k$--ample
bundles. Here is an example.

If $V$ is a $k$--ample vector bundle of rank $r$ on a complex
smooth projective variety $M$ of dimension $d$, then
$$
H^q(M, \, \Omega^p_M\otimes V) \,=\, 0
$$
if $p+q\,\geq\, d+r+k$ (Corollary 5.20 in page 96 of \cite{ShS}).
This yields the following:

Let $E_*$ be a parabolic $k$--ample vector bundle of rank $r$ on
$X$; as before, the underlying vector bundle is denoted by $E_0$.
Assume that $\lambda^i_1$ in \eqref{li} is nonzero
for all $i\,\in\, [1\, ,\ell]$. Then
$$
H^q(X, \, \Omega^p_X(\log D)\otimes E_0) \,=\, 0
$$
if $p+q\,\geq\, r+k+\dim X$ (see Theorem 4.4 in page 521 of
\cite{Bi2}). If some $\lambda^i_1$ are zero,
then $\Omega^p_X(\log D)\otimes E_0$ should be replaced
by the vector bundle $\overline{E}_p$ in Corollary 4.14
in page 527 of \cite{Bi2}.



\begin{thebibliography}{AAAA}

\bibitem[BBN]{BBN} V. Balaji, I. Biswas and D. S. Nagaraj,
Ramified $G$--bundles as parabolic bundles, \textit{Jour.
Ramanujan Math. Soc.} \textbf{18} (2003) 123--138.

\bibitem[Bi1]{Bi1} I. Biswas, Parabolic bundles as
orbifold bundles, \textit{Duke Math. Jour.}
\textbf{88} (1997) 305--325.

\bibitem[Bi2]{Bi2} I. Biswas, Parabolic ample bundles,
\textit{Math. Ann.} \textbf{307} (1997) 511--529.

\bibitem[Bi3]{Bi3} I. Biswas, Connections on a parabolic principal
bundle, Canadian Jour. Math. \textbf{58} (2006) 262--281.

\bibitem[Bi4]{Bi4} I. Biswas, Connections on a parabolic principal
bundle, II, \textit{Canadian Math. Bull.} \textbf{52} (2009) 175--185.

\bibitem[Ha]{Ha} R. Hartshorne, Ample vector bundles, \textit{Inst.
Hautes \'Etudes Sci. Publ. Math.} \textbf{29} (1966) 63--94.

\bibitem[MY]{MY} M. Maruyama and K. Yokogawa, Moduli of
parabolic stable sheaves, \textit{Math. Ann.} \textbf{293}
(1992) 77--99.

\bibitem[Se]{Se} C. S. Seshadri, Moduli of vector bundles on
curves with parabolic structures, \textit{Bull. Amer. Math.
Soc.} \textbf{83} (1977) 124--126.

\bibitem[ShS]{ShS} B. Shiffman and A. J. Sommese,
\textit{Vanishing theorems on complex manifolds}, Progress in
Mathematics, 56, Birkh\"auser, Boston, 1985.

\bibitem[So]{So} A. J. Sommese, Submanifolds of abelian
varieties, \textit{Math. Ann.} \textbf{233} (1978) 229--256.

\bibitem[Yo]{Yo} K. Yokogawa, Infinitesimal deformations of
parabolic Higgs sheaves, \textit{Int. Jour. Math.}
\textbf{6} (1995) 125--148.

\end{thebibliography}
\end{document}